\documentclass[12pt]{amsart}
\usepackage{tikz-cd}
\usetikzlibrary{matrix,arrows,decorations.pathmorphing}
\usepackage{amssymb}
\usepackage{amsmath,amscd}
\usepackage{mathrsfs}
\usepackage{url}
\usepackage{cite}
\usepackage{fullpage}
\usepackage{hyperref}
\usepackage{verbatim}
\usepackage{comment}
\usepackage{fancyvrb}
\usepackage{fvextra}
\usepackage{nccmath}
\newtheorem{thm}{Theorem}[section]
\newtheorem{prop}[thm]{Proposition}
\newtheorem{lem}[thm]{Lemma}
\newtheorem{cor}[thm]{Corollary}
\newtheorem{conj}[thm]{Conjecture}

\theoremstyle{definition}
\newtheorem{definition}[thm]{Definition}

\newtheorem{rem}[thm]{Remark}

\numberwithin{equation}{section}

\renewcommand{\L}{\mathcal{L}}

\newcommand{\X}{\mathcal{X}}

\newcommand{\T}{\mathcal{T}}
\newcommand{\V}{\mathcal{V}}

\newcommand{\N}{\mathcal{N}}

\newcommand{\zz}{\mathbb{Z}}
\newcommand{\qq}{\mathbb{Q}}

\newcommand{\M}{\mathcal{M}}
\newcommand{\p}{\mathbb{P}}
\newcommand{\pp}{\mathbb{P}}

\renewcommand{\H}{\mathcal{H}}
\newcommand{\F}{\mathcal{F}}
\renewcommand{\P}{\mathcal{P}}
\newcommand{\E}{\mathcal{E}}
\renewcommand{\O}{\mathcal{O}}

\newcommand{\Mg}{\mathcal{M}_g}
\newcommand{\Wmin}{\mathcal{W}_{\mathrm{min},2}}
\renewcommand{\tilde}{\widetilde}

\DeclareMathOperator{\ch}{ch}
\DeclareMathOperator{\td}{td}

\DeclareMathOperator{\SL}{SL}

\DeclareMathOperator{\Pic}{Pic}

\DeclareMathOperator{\Sym}{Sym}

\DeclareMathOperator{\BSL}{BSL}
\DeclareMathOperator{\BPGL}{BPGL}

\DeclareMathOperator{\BGM}{B\mathbb{G}_m}

\DeclareMathOperator{\W}{\mathcal{W}}
\renewcommand{\gg}{\mathbb{G}}

\newcommand{\bochao}[1]{{\color{red} ($\spadesuit$ Bochao: #1)}}

\begin{document}
\title{The Chow rings of moduli spaces of elliptic surfaces over $\p^1$}
\author{Samir Canning, Bochao Kong}
\email{srcannin@ucsd.edu}
\email{bokong@ucsd.edu}
\maketitle
\begin{abstract}
    Let $E_N$ denote the coarse moduli space of smooth elliptic surfaces over $\p^1$ with fundamental invariant $N$. We compute the Chow ring $A^*(E_N)$ for $N\geq 2$. For each $N\geq 2$, $A^*(E_N)$ is Gorenstein with socle in codimension $16$, which is surprising in light of the fact that the dimension of $E_N$ is $10N-2$. As an application, we show that the maximal dimension of a complete subvariety of $E_N$ is $16$. When $N=2$, the corresponding elliptic surfaces are K3 surfaces polarized by a hyperbolic lattice $U$. We show that the generators for $A^*(E_2)$ are tautological classes on the moduli space $\F_{U}$ of $U$-polarized K3 surfaces, which provides evidence for a conjecture of Oprea and Pandharipande on the tautological rings of moduli spaces of lattice polarized K3 surfaces.
\end{abstract}
\section{Introduction}
Given a smooth stack $X$ that is the solution to a moduli problem, there are often natural algebraic cycles called tautological classes in $A^*(X)$, the Chow ring of $X$ with rational coefficients. For example, when $X=\Mg$, the moduli space of smooth curves of genus $g$, there is the tautological subring $R^*(\Mg)\subset A^*(\Mg)$ generated by the $\kappa$-classes. Faber \cite{Faber} gave a series of conjectures on the structure of $R^*(\Mg)$, which assert that $R^*(\Mg)$ behaves like the algebraic cohomology ring of a smooth projective variety of dimension $g-2$, even though $\Mg$ is neither projective nor of dimension $g-2$. Looijenga \cite{Looijenga} proved that $R^i(\Mg)=0$ for $i>g-2$ and that $R^{g-2}(\Mg)\cong \qq$, settling one of Faber's conjectures. Looijenga's theorem gives a new proof of Diaz's result \cite{Diaz} that the maximal dimension of a complete subvariety of $\Mg$ is $g-2$. Faber further conjectured that $R^*(\Mg)$ should be a Gorenstein ring with socle in codimension $g-2$, meaning that the intersection product is a perfect pairing
\[
R^i(\Mg)\times R^{g-2-i}(\Mg)\rightarrow R^{g-2}(\Mg)\cong \qq.
\]
Faber \cite{Faber} and Faber--Zagier proved this conjecture for $g\leq 23$ by producing relations in the tautological ring and showing computationally that the resulting quotient is Gorenstein. 

Recently, there has been significant interest in the tautological rings $R^*(\F_{\Lambda})$ of the moduli spaces $\F_{\Lambda}$ of lattice polarized K3 surfaces \cite{MP,MOP,PY,BLMM,BerLi}. In \cite{MOP}, the tautological rings are defined as the subrings of $A^*(\F_{\Lambda})$ generated by the fundamental classes of Noether--Lefschetz loci together with push forwards of $\kappa$-classes from all Noether-Lefschetz loci. There are natural analogues of Faber's conjectures for $R^*(\F_{\Lambda})$.\footnote{We learned about these analogues from a lecture given by Rahul Pandharipande in the algebraic geometry seminar at UCSD and from a course on K3 surfaces given by Dragos Oprea.}
\begin{conj}[Oprea--Pandharipande]\label{conjecture}
Let $d=\dim \F_{\Lambda}$.
\begin{enumerate}
    \item For $i>d-2$, $R^i(\F_{\Lambda})=0$.
    \item There is an isomorphism $R^{d-2}(\F_{\Lambda})\cong \qq$.
\end{enumerate}
\end{conj}
The primary evidence for part (1) of this conjecture is a theorem of Petersen \cite[Theorem 2.2]{Petersen}, which says that the image $RH^{2*}(\F_{\Lambda})$ of $R^*(\F_{\Lambda})$ in cohomology under the cycle class map vanishes above cohomology degree $2(d-2)$. If Conjecture \ref{conjecture} holds, then 
one can further ask for the analogue of Faber's Gorenstein conjecture: is there a perfect pairing
\[
R^i(\F_{\Lambda})\times R^{d-2-i}(\F_{\Lambda})\rightarrow R^{d-2}(\F_{\Lambda})\cong \qq?
\]

In this paper, we study the Chow rings of moduli spaces $E_N$ of elliptic surfaces $Y$ fibered over $\p^1$ with section $s:\p^1\rightarrow Y$ and fundamental invariant $N$ (see Section 2 for definitions). 
The main result is that natural analogues of Faber's vanishing and Gorenstein conjectures hold for the entire Chow ring $A^*(E_N)$ for each $N\geq 2$.

\begin{thm}\label{main}
Let $N\geq 2$ be an integer. 
\begin{enumerate}
    \item The Chow ring has the form
\[
A^*(E_N)=\qq[a_1,c_2]/I_N
\]
where $a_1\in A^1(E_N), c_2\in A^2(E_N)$, and $I_N$ is the ideal generated by the two relations from Proposition \ref{relations}.
\item The Poincar\'e polynomial collecting dimensions of the Chow groups is given by
\begin{align*}
   p_N(t)&=\sum \dim A^i(E_N)t^i\\&=1+t+2t^2+2t^3+3t^4+3t^{5}+4t^{6}+4t^{7}+5t^{8}+\\ &\qquad +4t^{9}+4t^{10}+3t^{11}+3t^{12}+2t^{13}+2t^{14}+t^{15}+t^{16}.
 \end{align*}
 
\item The Chow ring $A^*(E_N)$ is Gorenstein with socle in codimension $16$. 
\end{enumerate}

\end{thm}
We also have similar partial results for Poincar\'e polynomial for the cohomology ring when $N=2$ that will appear in future work.

A notable property is that the dimensions of the Chow groups are independent of $N$. In particular, the Chow groups $A^i(E_N)$ are only nonzero in codimension $0\leq i \leq 16$, despite the fact that the dimensions of the moduli spaces $E_N$ go to infinity with $N$. Moreover, the ring structure depends in a simple and explicit way on $N$ coming from the relations in Proposition \ref{relations}. As a consequence of Theorem \ref{main}, we obtain an analogue of Diaz's theorem \cite{Diaz} on the maximal dimension of a complete subvariety of $\Mg$.  In our case, the bound is independent of $N$.

\begin{cor}\label{subvariety}
Let $N\geq 2$ be an integer. The maximal dimension of a complete subvariety of $E_N$ is $16$.
\end{cor}

When $N=2$, the corresponding elliptic surfaces are K3 surfaces polarized by a hyperbolic lattice $U$ with intersection matrix
\[
\begin{bmatrix}
0 & 1 \\
1 & 0
\end{bmatrix}.
\]
We show that the generators $a_1$ and $c_2$ of $A^*(E_2)$ have natural interpretations as tautological classes in $R^*(\F_{U})$, where $\F_U$ is the moduli space of $U$-polarized K3 surfaces.
\begin{thm}\label{taut}
Under the identification of $A^*(E_2)$ with $A^*(\F_{U})$, the classes $a_1$ and $c_2$ lie in $R^*(\F_{U})$. Therefore, $A^*(\F_U)=R^*(\F_U)$ is a Gorenstein ring with socle in codimension $16$.
\end{thm}
We view Theorem \ref{taut} as providing a piece of evidence toward Conjecture \ref{conjecture}.

The paper is structured as follows. In Section 2, we collect the necessary background on elliptic surfaces, the closely related notion of Weierstrass fibrations, and their moduli. In Section 3, we prove Theorem \ref{main} and Corollary \ref{subvariety}. In Section 4, we explore the case $N=2$ and prove Theorem \ref{taut}. We also compute relations among the codimension $1$ $\kappa$-classes.

\subsection*{Notations and Conventions}
\begin{enumerate}
    \item Schemes are over a fixed algebraically closed field $k$ of characteristic not $2$ or $3$. All stacks are fibered over the category of schemes over $k$.
    \item We denote the Chow ring of a space $X$ with \emph{rational coefficients} by $A^*(X)$. 
    \item We use the subspace (classical) convention for projective bundles.
\end{enumerate}

\subsection*{Acknowledgments} We are grateful to Kenneth Ascher, Elham Izadi, Hannah Larson, Dragos Oprea, Johannes Schmitt and David Stapleton for many helpful conversations. B.K. would like to especially thank Dragos Oprea for suggesting the study of the tautological ring for elliptic K3 surfaces. During the preparation of this article, S.C. was partially supported by NSF RTG grant DMS-1502651 and B.K. was partially supported by NSF grant DMS-1802228. This work will be part of the authors' Ph.D. theses.

\section{Elliptic Surfaces and Weierstrass Fibrations}
In this section, we collect the necessary background information on elliptic surfaces and Weierstrass fibrations following Miranda \cite{Miranda}. The main objects of interest in this paper will be moduli spaces of minimal elliptic surfaces over $\p^1$ with section.

\begin{definition}
A minimal elliptic surface over $\p^1$ with section consists of the following data:
\begin{enumerate}
    \item a smooth projective surface $Y$,
    \item a proper morphism $\pi:Y\rightarrow \p^1$ whose general fiber is a smooth connected curve of genus $1$ and such that none of the fibers contain any $(-1)$-curves,
    \item a section $s:\p^1\rightarrow Y$ of $\pi$. 
\end{enumerate}
\end{definition}
\begin{rem}
Note that the minimality condition is different from the usual one given in the birational geometry of surfaces. There can be $(-1)$-curves on the surface $Y$, but they must not lie in the fibers of $p$.
\end{rem}
We will study moduli spaces of minimal elliptic surfaces by studying the closely related notion of Weierstrass fibrations.
\begin{definition}
A Weierstrass fibration over $\p^1$ consists of the following data:
\begin{enumerate}
    \item a projective surface $X$,
    \item a flat proper morphism $p:X\rightarrow \p^1$ such that every fiber is an irreducible curve of arithmetic genus $1$ and the general fiber is smooth,
    \item a section $s:\p^1\rightarrow X$ of $p$ whose image does not intersect the singular points of any fiber.
\end{enumerate}
\end{definition}
Weierstrass fibrations $X\rightarrow \p^1$ have a natural invariant associated to them that governs aspects of the geometry of $X$ and the associated moduli spaces.
\begin{definition}
Let $p:X\rightarrow \p^1$ be a Weierstrass fibration.
\begin{enumerate}
    \item The fundamental line bundle associated to $p:X\rightarrow \p^1$ is the line bundle \[\mathbb{L}=(R^1p_*\O_X)^{\vee}.\]
    \item The fundamental invariant associated to $p:X\rightarrow \p^1$ is the integer 
    \[N=\deg \mathbb{L}.\]
\end{enumerate}
\end{definition}
Because $\mathbb{L}$ is a line bundle on $\p^1$, it is of the form $\O(N)$ where $N$ is the fundamental invariant. By \cite[Corollary 2.4]{Miranda}, the fundamental invariant is always nonnegative. 

There is a one-to-one correspondence between minimal elliptic surfaces with section and Weierstrass fibrations with at worst rational double points as singularities. Given a minimal elliptic surface $\pi:Y\rightarrow \p^1$, we obtain a Weierstrass fibration with at worst rational double points  $\p:X\rightarrow \p^1$ by contracting any rational components in the fibers that do not meet the section. Conversely, given a Weierstrass fibration $p:X\rightarrow \p^1$ with at worst rational double points as singularities, resolving the singularities and blowing down $(-1)$-curves in the fibers yields a minimal elliptic surface $\pi:Y\rightarrow \p^1$. We say that $Y$ contracts to $X$ and $X$ resolves to $Y$. Weierstrass fibrations have a representation as divisors on a $\p^2$-bundle over $\p^1$, which Miranda \cite{Miranda} used to construct coarse moduli spaces for Weierstrass fibrations, and hence elliptic surfaces, using Geometric Invariant Theory.

\begin{lem}[Corollary 2.5 of \cite{Miranda}]\label{conditions}
Let $\pi:Y\rightarrow \p^1$ be a minimal elliptic surface with section contracting to a Weierstrass fibration $p:X\rightarrow \p^1$ with fundamental invariant $N$. Then $X$ is isomorphic to the closed subscheme of $\p(\O\oplus\O(2N)\oplus\O(3N))$ defined by
\[
y^2z=x^3+Axz^2+Bz^3.
\]
where $A\in H^0(\p^1,\O(4N))$ and $B\in H^0(\p^1,\O(6N))$. Moreover,
\begin{enumerate}
    \item $4A^3+27B^2$ is not identically zero. If it vanishes at $q\in \p^1$, the fiber of $X$ over $q$ is singular.
    \item For every $q\in \p^1$, $v_q(A)\leq 3$ or $v_q(B)\leq 5$, where $v_q$ is the order of vanishing at $q$.
\end{enumerate}
\end{lem}
Set $V_{4N}:=H^0(\p^1,\O(4N))$ and $V_{6N}:=H^0(\p^1,\O(6N))$. Let $T_N\subset V_{4N}\oplus V_{6N}$ denote the open subspace satisfying conditions $(1)$ and $(2)$ from Lemma \ref{conditions}. The following is \cite[Corollary 2.8]{Miranda}.
\begin{cor}
The set of isomorphism classes of minimal elliptic surfaces $\pi:Y\rightarrow \p^1$ with $\deg R^1p_*\O_{X}=-N$ and with fixed section (equivalently, Weierstrass fibrations with only rational double points) is in $1-1$ correspondence with the set of orbits of $\SL_2\times \gg_m$ on $T_N$.
\end{cor}

In order to give the set of orbits a geometric structure, Miranda analyzes the stability of the action of $\SL_2\times \gg_m$ on $T_N$. 
\begin{prop}\label{stability}
Let $(A,B)\in V_{4N}\oplus V_{6N}$ be a pair of forms.
\begin{enumerate}
    \item The point corresponding to $(A,B)$ is not semistable if and only if there is a point $q\in \p^1$ such that
    \[
    v_q(A)>2N \textit{ and } v_q(B)>3N.
    \]
    \item The point corresponding to $(A,B)$ is not stable if and only if there is a point $q\in \p^1$ such that 
    \[
    v_q(A)\geq 2N \textit{ and } v_q(B)\geq 3N.
    \]
\end{enumerate}
\end{prop}
From Lemma \ref{conditions} and Proposition \ref{stability}, we see that as long as $N\geq 2$, points in $T_N$ are stable, and thus $E_N:=T_N//\SL_2\times \gg_m$ is a coarse moduli space for Weierstrass fibrations with fundamental invariant $N$. In particular, the natural morphism
\[
\E_N:=[T_N/\SL_2\times \gg_m]\rightarrow E_N
\]
from the quotient stack to the GIT quotient is a coarse moduli space morphism.

In Section 4, it will be useful for us to work on a stack $\W_N$ of Weierstrass fibrations with fundamental invariant $N$, not just the coarse moduli space constructed by Miranda. This stack is not the stack $\E_N$ defined above, but it is closely related as we will now explain. The stack $\W_N$ was recently defined in work of Park--Schmitt \cite{ParkSchmitt}, and we will briefly recall their construction.
\begin{definition}
Let $S$ be a scheme. A family of Weierstrass fibrations over $S$ is given by the data
\[
\mathcal{X}\xrightarrow{p} \P\xrightarrow{\gamma} S, \P\xrightarrow{s} \mathcal{X}
\]
where 
\begin{enumerate}
    \item $\gamma$ is a smooth, proper morphism locally of finite type, with geometric fibers isomorphic to $\p^1$,
    \item $p$ is a proper map with section $s$,
    \item the fibers $(\mathcal{X}_t\rightarrow \P_t, \P_t\rightarrow \X_t)$ on geometric points $t\in S$ are Weierstrass fibrations.
\end{enumerate}
\end{definition}
Park--Schmitt \cite{ParkSchmitt} define $\W$ to be the moduli stack whose objects over $S$ are families of Weierstrass fibrations over $S$ with morphisms over $T\rightarrow S$ given by fiber diagrams. The stack $\W_N$ is the open and closed substack parametrizing Weierstrass fibrations with fundamental invariant $N$. Finally, we consider the open substacks $\W_{\mathrm{min},N}\subset \W_{N}$ of Weierstrass fibrations satisfying the two conditions from Lemma \ref{conditions}. These stacks parametrize the Weierstrass fibrations with fundamental invariant $N$ that resolve to minimal elliptic surfaces. By \cite[Theorem 1.2]{ParkSchmitt}, the stacks $\W_{\mathrm{min},N}$ are smooth, separated Deligne-Mumford stacks for $N\geq 2$, and by \cite[Theorem 1.4]{ParkSchmitt}, $E_N$ is a coarse moduli space for $\W_{\mathrm{min},N}$

We now have three spaces of interest: $\E_N$, $\W_{\mathrm{min},N}$ and $E_N$. We want to compare their Chow rings. 
\begin{prop}\label{samechow}
The Chow rings of $\E_N$, $\W_{\mathrm{min},N}$ and $E_N$ are isomorphic.
\end{prop}
\begin{proof}
The space $E_N$ is a coarse moduli space for both stacks $\E_N$ and $\W_{\mathrm{min},n}$. Therefore, since we are using rational coefficients, all three Chow rings are isomorphic by a result of Vistoli \cite[Proposition 6.1]{Vistoli}.    
\end{proof}
\begin{rem}
The difference between the stacks $\W_{\mathrm{min},N}$ and $\E_N$ is that $\E_N$ is a $\mu_2$-banded gerbe over $\W_{\mathrm{min},N}$. The gerbe structure arises from the map $\BSL_2\rightarrow \BPGL_2$.
\end{rem}
\section{Computing the Chow ring}

By Proposition \ref{samechow}, it suffices to compute $A^*(\E_N)$ in order to prove Theorem \ref{main}. Let $\Delta_N\subset V_{4N}\oplus V_{6N}$ denote the complement of $T_N$. We have the excision exact sequence 
\begin{equation}\label{excision}
    A_{*}([\Delta_N/\SL_2\times \gg_m])\rightarrow A^*([V_{4N}\oplus V_{6N}/\SL_2\times \gg_m])\rightarrow A^*(\E_N)\rightarrow 0.
\end{equation}
We want to study the image of $A_*([\Delta_N/\SL_2\times \gg_m])$ in $A^*([V_{4N}\oplus V_{6N}/\SL_2\times \gg_m])$. 

We begin with background information on the stack $[V_{4N}\oplus V_{6N}/\SL_2\times \gg_m]$. The stack $\BSL_2$ is the classifying stack for rank $2$ vector bundles with trivial first Chern class. Let $\V$ denote the universal rank $2$ vector bundle with trivial first Chern class over $\BSL_2$. Set $c_2:=c_2(\V)$. Similarly, the stack $\BGM$ is the classifying stack for line bundles. Let $\M$ denote the universal line bundle over $\BGM$. Set $a_1:=c_1(\M)$. By abuse of notation, we will not distinguish between $\V$, $\M$, $c_2$, and $a_1$ and their pullbacks to the product $\BSL_2\times \BGM$ under the natural projection maps. We will interpret the stack $\BSL_2\times \BGM$ as the stack of line bundles of relative degree $N$ on $\p^1$-bundles as in \cite{Larson} as follows. Consider the universal $\p^1$-bundle 
\[
\gamma:\p(\V)\rightarrow \BSL_2\times \BGM.
\]
Fix $N\geq 0$ and set $\L:=\gamma^*\M(N)$, the universal relative degree $N$ line bundle on $\p(\V)$.
\begin{lem}\label{totspace}\leavevmode
\begin{enumerate}
    \item The stack $[V_{4N}\oplus V_{6N}/\SL_2\times \gg_m]$ is the total space of the vector bundle $\gamma_*(\L^{\otimes 4}\oplus \L^{\otimes 6})$ on $\BSL_2\times \BGM$.
    \item There is an isomorphism of graded rings
    \[
    A^*([V_{4N}\oplus V_{6N}/\SL_2\times \gg_m])\cong \qq[a_1,c_2],
    \]
    with $a_1$ in degree $1$ and $c_2$ in degree $2$.
\end{enumerate}
\end{lem}
\begin{proof}
Part (1) follows from cohomology and base change. Indeed, the fibers of $\gamma_*(\L^{\otimes 4}\oplus \L^{\otimes 6})$ are canonically identified with $V_{4N}\oplus V_{6N}$, and the higher cohomology vanishes. For part (2), we note that by part (1) and the homotopy property for Chow rings, there is an isomorphism
\[
A^*([V_{4N}\oplus V_{6N}/\SL_2\times \gg_m])\cong A^*(\BSL_2\times \BGM).
\]
A standard calculation in equivariant intersection theory \cite[Section 15]{Totaro} shows that \[
A^*(\BSL_2\times \BGM)\cong \qq[a_1,c_2]\] as graded rings.
\end{proof}

\subsection{Computing the ideal of relations}
By Lemma \ref{totspace}, the exact sequence \eqref{excision} can be rewritten as
\begin{equation}
    A_*([\Delta_N/\SL_2\times \gg_m])\rightarrow \qq[a_1,c_2]\rightarrow A^*(\E_N)\rightarrow 0.
\end{equation}
It follows that $A^*(\E_N)$, and hence $A^*(E_N)$, is a quotient of $\qq[a_1,c_2]$ by the ideal $I_N$ generated by the image of $A_*([\Delta_N/\SL_2\times \gg_m])$.

Lemma \ref{conditions} tells us exactly when a pair $(A,B)\in V_{4N}\oplus V_{6N}$ is contained in $\Delta_N$. We write $\Delta_N=\Delta_N^{1}\cup \Delta_{N}^{2}$, where $\Delta_{N}^{1}$ parametrizes the pairs of forms $(A,B)$ such that $4A^3+27B^2$ is identically zero (corresponding to Lemma \ref{conditions} part (1)), and $\Delta_{N}^{2}$ parametrizing pairs of forms $(A,B)$ such that $v_q(A)\geq 4$ or $v_q(B)\geq 6$ for some point $p\in \p^1$ (corresponding to Lemma \ref{conditions} part (2)). First, we will determine the relations obtained from excising the pairs $(A,B)\in \Delta_{N}^{2}$. To do so, we need to introduce bundles of principal parts. We will follow the treatment in \cite{EH}.

Let $b:Y\rightarrow Z$ be a smooth proper morphism. Let $\Delta_{Y/Z}\subset Y \times_Z Y$ be the relative diagonal. With $p_1$ and $p_2$ the projection maps, we obtain the following commutative diagram:
\[
\begin{tikzcd}
\Delta_{Y/Z} \arrow[rd] \arrow[rrd, bend left] \arrow[rdd, bend right] &                                      &                  \\
                                                                 & Y \times_{Z} Y \arrow[d, "p_2"'] \arrow[r, "p_1"] & Y \arrow[d, "b"] \\
                                                                 & Y \arrow[r, "b"']                    & Z.               
\end{tikzcd}
\]
\begin{definition}
Let $\F$ be a vector bundle on $Y$ and let $\mathcal{I}_{\Delta_{Y/Z}}$ denote the ideal sheaf of the diagonal in $Y \times_Z Y$. The bundle of relative $m^{\text{th}}$ order principal parts $P^m_{b}(\V)$ is defined as
\[
P^m_{b}(\F)=p_{2*}(p_1^*\F\otimes \mathcal{O}_{Y \times_Z Y}/\mathcal{I}_{\Delta_{Y/Z}}^{m+1}).
\]
\end{definition}
The following explains all the basic properties of bundles of principal parts that we need.
\begin{prop}[Theorem 11.2 in \cite{EH}]\label{parts}
With notation as above,
\begin{enumerate}
    \item There is an isomorphism $b^*b_*\F\xrightarrow{\sim} p_{2*}p_1^*\F$.
    \item The quotient map $p_1^*\F\rightarrow p_1^*\F\otimes \mathcal{O}_{Y \times_Z Y}/\mathcal{I}_{\Delta_{Y/Z}}^{m+1}$ pushes forward to a map
    \[
    b^*b_*\F\cong p_{2*}p_1^*\F\rightarrow P^{m}_{b}(\F),
    \]
    which, fiber by fiber, associates to a global section $\delta$ of $\F$ a section $\delta'$ whose value at $z\in Z$ is the restriction of $\delta$ to an $m^{\text{th}}$ order neighborhood of $z$ in the fiber $b^{-1}b(z)$.
    \item \label{filtration} $P^0_{b}(\F)=\F$. For $m>1$, the filtration of the fibers $P^m_{b}(\F)_y$ by order of vanishing at $y$ gives a filtration of $P^m_{b}(\F)$ by subbundles that are kernels of the natural surjections
    $P^m_{b}(\F)\rightarrow P^k_{b}(\F)$ for $k<m$. The graded pieces of the filtration are identified by the exact sequences
    \[
    0\rightarrow \F\otimes \Sym^m(\Omega_{Y/Z})\rightarrow P^m_{b}(\F)\rightarrow P^{m-1}_{b}(\F)\rightarrow 0.
    \]
\end{enumerate}
\end{prop}

By $(2)$ of Proposition \ref{parts}, there is a morphism
\[
\psi: \gamma^*\gamma_*(\L^{\otimes 4}\oplus \L^{\otimes 6})\rightarrow P^3_{\gamma}(\L^{\otimes 4})\oplus P^5_{\gamma}(\L^{\otimes 6})
\]
which, along points in the $\p^1$ fibers, sends $A$ (respectively, $B$) to a third (respectively, fifth) order neighborhood. The kernel of this map therefore parametrizes the triples $(A,B,q)$ such that $v_q(A)\geq 4$ and $v_q(B)\geq 6$. Looking fiber-by-fiber, one sees that the map $\psi$ is surjective. Therefore, the kernel $K$ of $\psi$ is a vector bundle. We obtain the following commutative diagram where $\phi, \phi'$ and $\phi''$ are vector bundle morphisms.
\begin{equation}\label{commutative}
\begin{tikzcd}
K \arrow[r, "i", hook] \arrow[rd, "\phi''"'] & \gamma^*\gamma_*(\L^{\otimes 4}\oplus \L^{\otimes 6}) \arrow[r, "\gamma'"] \arrow[d, "\phi'"'] & \gamma_*(\L^{\otimes 4}\oplus \L^{\otimes 6}) \arrow[d, "\phi"] \\
                                             & \p(\V) \arrow[r, "\gamma"]                                                                       & \BSL_2\times \BGM.                                              
\end{tikzcd}
\end{equation}
By construction, $K$ maps properly and surjectively onto $[\Delta_{N}^{2}/\SL_2\times \gg_m]$ under the identification of $\gamma_*(\L^{\otimes 4}\oplus \L^{\otimes 6})$ with $[V_{4N}\oplus V_{6N}/\SL_2\times\gg_m]$ from Lemma \ref{totspace}. Consequently, the images of the push forward maps
\[
\gamma'_*i_*:A_*(K)\rightarrow A^*(\gamma_*(\L^{\otimes 4}\oplus \L^{\otimes 6}))=A^*([V_{4N}\oplus V_{6N}/\BSL_2\times \BGM])
\]
and
\[
A_{*}([\Delta^{2}_N/\SL_2\times \gg_m])\rightarrow A^*([V_{4N}\oplus V_{6N}/\SL_2\times \gg_m])
\]
are the same.
\begin{prop}\label{relations}
Let $z$ denote the hyperplane class of $\p(\V)$.
The image of the push forward map $\gamma'_*i_*:A^*(K)\rightarrow A^*(\gamma_*(\L^{\otimes 4}\oplus \L^{\otimes 6}))$ is the ideal generated by the two classes
\begin{enumerate}
    \item $\phi^*\gamma_*(c_{\text{top}}(P^3_{\gamma}(\L^{\otimes 4})\oplus P^5_{\gamma}(\L^{\otimes 6})))$, and
    \item $\phi^*\gamma_*(c_{\text{top}}(P^3_{\gamma}(\L^{\otimes 4})\oplus P^5_{\gamma}(\L^{\otimes 6}))\cdot z)$.
\end{enumerate}
\end{prop}
\begin{proof}
Let $\alpha\in A^*(K)$. Then because $K$ is a vector bundle over $\p(\V)$, we see that $\alpha=\phi''^*(\beta)$ for some class $\beta\in A^*(\p (\V))$, so we have
\[
\alpha=\phi''^*(\beta)=i^*\phi'^*(\beta).
\]
Pushing forward, we obtain
\[
\gamma'_*i_*\alpha=\gamma'_*i_*i^*\phi'^*(\beta)=\gamma'_*([K]\cdot \phi'^*\beta).
\]
Because $K$ is the kernel of the vector bundle morphism 
\[
\psi:\gamma^*\gamma_*(\L^{\otimes 4}\oplus \L^{\otimes 6})\rightarrow P^3_{\gamma}(\L^{\otimes 4})\oplus P^5_{\gamma}(\L^{\otimes 6}),
\]
the fundamental class $[K]$ is given by $\phi'^*(c_{\text{top}}(P^3_{\gamma}(\L^{\otimes 4})\oplus P^5_{\gamma}(\L^{\otimes 6})))$. Because the square in the commutative diagram \eqref{commutative} is Cartesian, $\gamma'_*\phi'^*=\phi^*\gamma_*$, so 
\[
\gamma'_*i_*\alpha=\phi^*\gamma_*(c_{\text{top}}(P^3_{\gamma}(\L^{\otimes 4})\oplus P^5_{\gamma}(\L^{\otimes 6}))\cdot \beta).
\]
Because $\p (\V)$ is a projective bundle, $\beta$ can be written as
\[
\beta=\gamma^*\beta_1+\gamma^*\beta_2z,
\]
where $\beta_1$ and $\beta_2$ are classes in $A^*(\BSL_2\times \BGM)$. The statement of the proposition follows.
\end{proof}
\begin{rem}
The relations from Proposition \ref{relations} can be computed explicitly as polynomials of $a_1$, $c_2$, and $N$ using the splitting principle and Proposition \ref{parts}. We carried out this computation in Macaulay2 \cite{M2} using the package Schubert2 \cite{S2}.

\begin{equation*}\medmath{\begin{split}
        \phi^*\gamma_*(c_{\text{top}}(P^3_{\gamma}(\L^{\otimes 4})\oplus P^5_{\gamma}(\L^{\otimes 6})))&=119439360N^9c_2^4a_1-859963392N^8c_2^4a_1-1433272320N^7c_2^3a_1^3\\&+2598469632N^7c_2^4a_1
        +8026324992N^6c_2^3a_1^3+3009871872N^5c_2^2a_1^5\\&-4277919744N^6c_2^4a_1-18189287424N^5c_2^3a_1^3-12039487488N^4c_2^2a_1^5\\&-1433272320N^3c_2a_1^7+4164009984N^5c_2^4a_1+21389598720N^4c_2^3a_1^3\\&+18189287424N^3c_2^2a_1^5+3439853568N^2c_2a_1^7+119439360Na_1^9\\&-2427125760N^4c_2^4a_1-13880033280N^3c_2^3a_1^3-12833759232N^2c_2^2a_1^5\\&-2598469632Nc_2a_1^7-95551488a_1^9+813809664N^3c_2^4a_1\\&+4854251520N^2c_2^3a_1^3+4164009984Nc_2^2a_1^5+611131392c_2a_1^7\\&-139567104N^2c_2^4a_1-813809664Nc_2^3a_1^3-485425152c_2^2a_1^5\\&+8847360Nc_2^4a_1+46522368c_2^3a_1^3.
\end{split}}
\end{equation*}

\begin{equation*}\medmath{
    \begin{split}
        \phi^*\gamma_*(c_{\text{top}}(P^3_{\gamma}(\L^{\otimes 4})\oplus P^5_{\gamma}(\L^{\otimes 6}))\cdot z)&=-11943936N^{10}c_2^5+95551488N^9c_2^5+537477120N^8c_2^4a_1^2\\&-324808704N^8c_2^5-3439853568N^7c_2^4a_1^2-2508226560N^6c_2^3a_1^4\\&+611131392N^7c_2^5+9094643712N^6c_2^4a_1^2+12039487488N^5c_2^3a_1^4\\&+2508226560N^4c_2^2a_1^6-694001664N^6c_2^5-12833759232N^5c_2^4a_1^2\\&-22736609280N^4c_2^3a_1^4-8026324992N^3c_2^2a_1^6-537477120N^2c_2a_1^8\\&+485425152N^5c_2^5+10410024960N^4c_2^4a_1^2+21389598720N^3c_2^3a_1^4\\&+9094643712N^2c_2^2a_1^6+859963392Nc_2a_1^8+11943936a_1^{10}\\&-203452416N^4c_2^5-4854251520N^3c_2^4a_1^2-10410024960N^2c_2^3a_1^4\\&-4277919744Nc_2^2a_1^6-324808704c_2a_1^8+46522368N^3c_2^5\\&+1220714496N^2c_2^4a_1^2+2427125760Nc_2^3a_1^4+694001664c_2^2a_1^6\\&-4423680N^2c_2^5-139567104Nc_2^4a_1^2-203452416c_2^3a_1^4+4423680c_2^4a_1^2.
    \end{split}}
\end{equation*}
Simplifying, we have that the ideal that these two classes generate is the ideal generated by the following two polynomials, $p_1$ and $p_2$.
\begin{equation*}
\medmath{
    \begin{split}
     p_1&=(1620N-1296)a_1^9+(-19440N^3+46656N^2-35244N+8289)a_1^7c_2\\&+(40824N^5-163296N^4+246708N^3-174069N^2+56478N-6584)a_1^5c_2^2\\&+(-19440N^7+108864N^6-246708N^5+290115N^4-188260N^3+65840N^2-11038N+631)a_1^3c_2^3\\&+(1620N^9-11664N^8+35244N^7-58023N^6+56478N^5-32920N^4+11038N^3-1893N^2+120N)a_1c_2^4.
    \end{split}}
\end{equation*}
\begin{equation*}
\medmath{
\begin{split}
    p_2&=324a_1^{10}+(-14580N^2+23328N-8811)a_1^8c_2\\&+(68040N^4-217728N^3+246708N^2-116046N+18826)a_1^6c_2^2\\&+(-68040N^6+326592N^5-616770N^4+580230N^3-282390N^2+65840N-5519)a_1^4c_2^3\\&+(14580N^8-93312N^7+246708N^6-348138N^5+282390N^4-131680N^3+33114N^2-3786N+120)a_1^2c_2^4\\&+(-324N^{10}+2592N^9-8811N^8+16578N^7-18826N^6+13168N^5-5519N^4+1262N^3-120N^2)c_2^5.
\end{split}}
\end{equation*}

\end{rem}
\begin{lem}\label{codimension}
The codimension of $\Delta_N^1$ in $V_{4N}\oplus V_{6N}$ is $8N+1$.
\end{lem}
\begin{proof}
Let $t$ be an affine coordinate on $\p^1$. Then we can factor $A(t)$ and $B(t)$ into linear factors as
\[
A(t)=a\prod_{i=1}^{4N}(t-c_i) \textit{ and } B(t)=b\prod_{i=1}^{6N}(t-d_i).
\]
Because $4A^3+27B^2$ is identically zero, we have the equation
\[
4a^3\prod_{i=1}^{4N}(t-c_i)^3=-27b^2\prod_{i=1}^{6N}(t-d_i)^2.
\]
By comparing the orders of vanishing of each side, we see that $A(t)=aG(t)^2$ and $B(t)=bG(t)^3$, where $G$ is a polynomial of degree $2N$ and $4a^3+27b^2=0$. It follows that the codimension of $\Delta^1_{N}$ is given by
\[
\dim (V_{4N}\oplus V_{6N})-\dim V_{2N}=10N+2-2N-1=8N+1.
\]
\end{proof}
We can now complete the proof of Theorem \ref{main}.
\begin{proof}[Proof of Theorem \ref{main}]
By a calculation in Macaulay2 \cite{M2}, the graded ring $\qq[a_1,c_2]/I_N$ vanishes in degree $17$ and higher, where $I_N$ is the ideal generated by the relations from Proposition \ref{relations}. We have the excision exact sequence
\[
A_*([\Delta^{1}_N/\SL_2\times \gg_m])\rightarrow \qq[a_1,c_2]/I_N\rightarrow A^*(\E_N)\rightarrow 0. 
\]
By Lemma \ref{codimension}, the image of 
\[
A_*([\Delta^{1}_N/\SL_2\times \gg_m])\rightarrow \qq[a_1,c_2]/I_N
\]
lies in codimension $17$ or higher, so it is identically zero. Therefore, 
\[
\qq[a_1,c_2]/I_N\cong A^*(\E_N).
\]
This completes the proof of Theorem \ref{main} part (1). Parts (2) and (3) are consequences of part (1) together with a computation in Macaulay2 \cite{M2} that computes the Hilbert Series of the ring $\qq[a_1,c_2]/I_N$ and verifies that the intersection pairing is perfect.
\end{proof}
\begin{proof}[Proof of Corollary \ref{subvariety}]
Miranda's construction of $E_N$ by geometric invariant theory \cite{Miranda} shows that $E_N$ is a quasi-projective variety. It thus admits an ample line bundle $L$. If $S$ is a complete subvariety of dimension $d$, then, because $L$ is ample,
\[
c_1(L)^d\cdot S>0.
\]
Hence, $c_1(L)^d$ is numerically nonzero. By Theorem \ref{main}, it follows that $d\leq 16$. 
\end{proof}
\section{The Tautological Ring}
\subsection{Stacks of lattice polarized K3 surfaces}\label{stacks}
Let $\Lambda \subset U^{\oplus 3}\oplus E_8(-1)^{\oplus 2}$ be a fixed rank $r$ primitive sublattice with signature $(1,r-1)$, and let $v_1,\dots,v_r$ be an integral basis of $\Lambda$. A $\Lambda$-polarization on a K3 surface $X$ is a primitive embedding 
\[
j:\Lambda \hookrightarrow \Pic(X)
\]
such that
\begin{enumerate}
    \item The lattices $H^2(X,\zz)$ and $U^{\oplus 3}\oplus E_8(-1)^{\oplus 2}$ are isomorphic via an isometry restricting to the identity on $\Lambda$, where we view $\Lambda$ as sitting inside $H^2(X,\zz)$ via $\Lambda\hookrightarrow \Pic(X)\hookrightarrow H^2(X,\zz)$.
    \item The image of $j$ contains the class of a quasi-polarization.
\end{enumerate}
Beauville \cite{Beauville} constructed moduli stacks $\F_{\Lambda}$ of $\Lambda$-polarized K3 surfaces, and showed that they are smooth Deligne--Mumford stacks of dimension $19-r$. Using the surjectivity of the period map, one can construct coarse moduli spaces $F_{\Lambda}$ for $\F_{\Lambda}$ \cite{Dolgachev}.

We think of the stacks $\F_{\Lambda}$ as parametrizing families of K3 surfaces
\[
\pi:X\rightarrow S
\]
together with $r$ line bundles $H_1,\dots, H_r$ on $X$ corresponding to the basis $v_1,\dots,v_r$ of $\Lambda$, well-defined up to pullbacks from $\Pic(S)$. Technically, these bundles exist only \'etale locally, as they are defined as sections of the sheaf $\Pic_{X/S}$, which is the \'etale sheafification of the presheaf on the category of schemes over $S$
\[
T\mapsto \Pic(X_T)/\Pic(T).
\]
We will generally suppress this detail, but we will remark when it is important. 
There are forgetful morphisms 
\[
\F_{\Lambda'}\hookrightarrow \F_{\Lambda}
\]
for any lattice $\Lambda\subset \Lambda'$. When $\Lambda$ is strictly contained in $\Lambda'$, we call the subvarieties $\F_{\Lambda'}$ \emph{Noether-Lefschetz loci} of $\F_{\Lambda}$.
\subsection{The tautological ring of $\F_{\Lambda}$}
The stack $\F_{\Lambda}$ comes equipped with a universal K3 surface
\[
\pi_{\Lambda}:\X_{\Lambda}\rightarrow \F_{\Lambda}.
\]
and universal bundles $\H_1,\dots \H_r$, well-defined up to pullbacks from $\F_{\Lambda}$. Let $\T_{\pi_{\Lambda}}$ denote the relative tangent bundle. Following \cite{MOP}, we define the $\kappa$-classes
\[
\kappa^{\Lambda}_{a_1,\dots,a_r,b}:=\pi_{\Lambda*}\left(c_1(\H_1)^{a_1}\cdots c_1(\H_r)^{a_r}\cdot c_2(\T_{\pi_{\Lambda}})^b\right).
\]
\begin{definition}
The tautological ring $R^*(\F_{\Lambda})$ is the subring of $A^*(\F_{\Lambda})$ generated by pushforwards from the Noether--Lefschetz loci of all $\kappa$-classes.
\end{definition}
By \cite{Borcherds} or \cite{FarRim}, the Hodge class $\lambda:=c_1(\pi_{\Lambda*}\omega_{\pi_{\Lambda}})$ lies in the tautological ring $R^*(\F_{\Lambda})$ for all $\Lambda$, as it is supported on Noether--Lefschetz divisors.
\subsection{Moduli of elliptic K3 surfaces and Weierstrass fibrations}
Let $p:X\rightarrow \p^1$ be a minimal elliptic surface over $\p^1$ with fundamental invariant $2$. Then $X$ is a K3 surface, and the class of the fiber $f$ and section $\sigma$ form a primitively embedded lattice $U\subset \Pic(X)$ equivalent to a hyperbolic lattice, whose image contains a quasi-polarization $\sigma+2f$. Conversely, given a K3 surface $X$, a primitive embedding of a hyperbolic lattice $U\hookrightarrow \Pic(X)$ whose image contains a quasi-polarization allows one to define a morphism $p:X\rightarrow \p^1$ with section $s:\p^1\rightarrow X$ with fundamental invariant $2$ \cite[Theorem 2.3]{ClingherDoran}.  Because of this, we call the stack $\F_{U}$ the stack parametrizing elliptic K3 surfaces with section. By \cite[Theorem 7.9]{OdakaOshima}, the coarse moduli space $F_U$ is isomorphic to $E_2$. By the discussion in subsection \ref{stacks}, $\F_U$ comes equipped with a universal K3 surface and two universal line bundles
\[
\pi_U:\X_U\rightarrow \F_U, \quad \O(f)\rightarrow \X_U, \quad \O(\sigma)\rightarrow \X_U.
\]
The intersection matrix of $\O(\sigma)$ and $\O(f)$ is 
\[\begin{bmatrix}
\O(\sigma)^2 & \O(\sigma)\cdot \O(f) \\
\O(\sigma)\cdot \O(f) & \O(f)^2
\end{bmatrix}=\begin{bmatrix}
-2 & 1 \\
1 & 0
\end{bmatrix},\]
which can be obtained by a change of basis from the usual intersection matrix for a hyperbolic lattice $U$:
\[
\begin{bmatrix}
0 & 1 \\
1 & 0
\end{bmatrix}.
\]
We prefer to take $\O(f)$ and $\O(\sigma)$ as our basis because of their geometric meaning. 
Recall that the stack $\Wmin$ parametrizes families of Weierstrass fibrations resolving to minimal elliptic surfaces. We will construct a morphism
\[
G:\F_U\rightarrow \Wmin,
\]
which is a relative version of the morphism sending an elliptic K3 surface to its associated Weierstrass fibration.
Let $\pi:X\rightarrow S$ be a family of $U$-polarized K3 surfaces, equipped with bundles $\O(f)$ and $\O(\sigma)$ on $X$, up to an \'etale cover of $S$. The surjection
\[
\pi^*\pi_*\O(f)\rightarrow \O(f)
\]
defines a morphism
\[
p:X\rightarrow \p(\pi_*\O(f)^{\vee})
\]
over $S$. The relative effective Cartier divisor associated to $\O(\sigma)$ allows us to define a section $s$ of $p$. The surjection
\[
p^*p_*\O(3\sigma)\rightarrow \O(3\sigma)
\]
defines a morphism $i:X\rightarrow \p(p_*\O(3\sigma)^{\vee})$. Let $Y$ denote the image of $X$ under $i$. Then $Y$ is a family of Weierstrass fibrations over $S$. This construction defines the morphism
\[
G:\F_{U}\rightarrow \Wmin.
\]

\begin{rem}
We note that in constructing $Y$, we chose line bundles $\O(f)$ and $\O(\sigma)$. Technically, we could only do so \'etale locally. The projective bundle $\p(\pi_*\O(f)^{\vee})\rightarrow S$ will only descend to a smooth proper morphism, locally of finite type, with geometric fibers isomorphic to $\p^1$: it will not necessarily be the projectivization of a vector bundle on $S$. Second, even once we pass to an \'etale cover, $\O(f)$ and $\O(\sigma)$ are only defined up to pullbacks from $\Pic(S)$. If we made different choices for $\O(f)$ and $\O(\sigma)$ the resulting Weierstrass fibration would be canonically isomorphic to the original one because for any vector bundle $\E$ and line bundle $\L$, $\p(\E\otimes \L)$ is canonically isomorphic to $\p(\E)$. 
\end{rem}

Consider the following Cartesian diagram, which defines the stack $\tilde{\F}_U$.
\[
\begin{tikzcd}
\tilde{\F}_U \arrow[d] \arrow[r, "G'"] & \E_2 \arrow[d] \\
\F_U \arrow[r, "G"]                    & \Wmin         
\end{tikzcd}
\]
The vertical morphisms are $\mu_2$-banded gerbes. In fact, we can explicitly describe the functor of points for $\tilde{\F}_U$. A morphism from a scheme $S$ to $\tilde{\F}_U$ is a family 
\[
(\pi:X\rightarrow S, \O(f), \O(\sigma), \N)
\]
where $(\pi:X\rightarrow S, \O(f), \O(\sigma))$ is a family of $U$-polarized K3 surfaces and $\N$ is a line bundle on $S$ such that
\[
\N^{\otimes 2}\cong \det \pi_*\O(f).
\]
Recall that $\E_2$ has a universal rank $2$ vector bundle with trivial first Chern class $\V$ and a universal line bundle $\M$. By construction of the map $G$ and its base change $G'$, we have that
\[
G'^*\V=\pi_*\O(f)^{\vee}\otimes \N, 
\]
where $\N$ is the universal square root of $\det \pi_*\O(f)$. 
We will abuse notation and denote the universal K3 surface on $\F_U$ and $\tilde{\F}_U$ both by $\pi$.
\begin{lem}\label{c2}
The class $c_2(\pi_*\O(f)^{\vee}\otimes \N)$ on $\tilde{\F}_U$ is the pullback of a tautological class on $\F_U$.
\end{lem}
\begin{proof}
Note that 
\begin{equation*}
    \begin{split}
        c_2(\pi_*\O(f)^{\vee}\otimes \N)&=c_1(\N)^2+c_1(\pi_*\O(f)^{\vee})c_1(\N)+c_2(\pi_*\O(f)^{\vee})\\
        &=\frac{1}{4}c_1(\det \pi_*\O(f))^2-\frac{1}{2}c_1(\pi_*\O(f))c_1(\det\pi_*\O(f))+c_2(\pi_*\O(f))\\
        &=-\frac{1}{4}c_1(\pi_*\O(f))^2+c_2(\pi_*\O(f)).
 \end{split}
\end{equation*}
It thus suffices to show that the Chern classes of $\pi_*\O(f)$ are tautological. By Grothendieck--Riemann--Roch, we have
\[
\ch(\pi_! \O(f))=\pi_*(\ch(\O(f))\cdot \td(T_{\pi})).
\]
By definition, the classes on the right hand side are tautological. We note that
\[
\pi_!\O(f)=\pi_*\O(f)
\]
because $\pi$ is a relative K3 surface. By comparing degree $1$ parts of both sides, we see that $c_1(\pi_*\O(f))$ is tautological. By comparing degree $2$ parts, we see that $c_2(\pi_*\O(f))$ is tautological.
\end{proof}
\begin{proof}[Proof of Theorem \ref{taut}]
Each of the stacks $\E_2$, $\Wmin$, $\F_U$, and $\tilde{\F}_U$ has the same coarse moduli space $E_2$. They thus all have isomorphic Chow rings, and proper push forward $A_*(Z)\rightarrow A_*(E_2)$ is an isomorphism of Chow groups, where $Z$ is any of the four stacks above \cite[Proposition 6.1]{Vistoli}. By Theorem \ref{main}, $A^1(E_2)$ is generated by the push forward of $a_1$. By \cite[Theorem 2.1]{Petersen} or the proof of \cite[Corollary 4.2]{vanderGeerKatsura}, the tautological class $\lambda$ is nonvanishing on $\F_{U}$. It follows that $A^1(\F_U)$ is generated by $\lambda$, so $A^1(\F_U)=R^1(\F_U)$. By Theorem \ref{main}, $A^2(E_2)$ is generated by the push forwards of $a_1^2$ and $c_2$.  By Lemma \ref{c2}, the class $c_2$ pulls back to a class in $A^2(\tilde{\F}_U)$ that is the pullback of a tautological class from $A^2(\F_U)$. It follows that $A^2(\F_U)=R^2(\F_U)$, as the images of $a_1^2$ and $c_2$ in $A^2(E_2)$ can both be obtained by pushing forward tautological classes from $\F_U$. Therefore, $A^*(\F_U)=R^*(\F_U)$. The fact that $A^*(\F_U)=R^*(\F_U)$ is Gorenstein with socle in codimension $16$ follows from Theorem \ref{main}.
\end{proof}

\subsection{Codimension one classes} By Theorems \ref{main} and \ref{taut}, $A^1(\F_U)$ is of rank one and the Hodge class $\lambda$ is a generator. It is natural to ask how to represent $\kappa$-classes explicitly in terms of the Hodge class $\lambda$.

\begin{prop}\label{div}
The following four linear combinations of $\kappa$-classes are independent of the choice of universal line bundles.  Moreover, they are all multiples of the Hodge class $\lambda$.
\[\kappa_{3,0,0}+\frac{1}{4}\kappa_{1,0,1}=\frac{7}{2}\lambda,\quad 3\kappa_{2,1,0}-\frac{1}{4}\kappa_{1,0,1}+\frac{1}{4}\kappa_{0,1,1}=\frac{1}{2}\lambda,\]
\[3\kappa_{1,2,0}-\frac{1}{4}\kappa_{0,1,1}=-3\lambda,\quad \kappa_{0,3,0}=0.\]
where $\kappa_{i,j,k}:=\pi_*\left(c_1(\O(\sigma))^i\cdot c_1(\O(f))^{j}\cdot c_2(\T_{\pi})^k\right)$. 
\end{prop}

\begin{proof}
A direct computation shows the above four $\kappa$ combinations are invariant under $f\mapsto f+\pi^*(l)$ and $\sigma \mapsto \sigma+\pi^*(l^{\prime})$ for any $l,l^{\prime}\in A^1(\F_U)$.

By Theorem \ref{main}, we know $A^1(\F_U)$ is of rank one, so it is sufficient to check the identities by computing their intersection numbers with a suitable test curve:
\[
\iota: C\rightarrow \F_U.
\]
To construct the curve, we use the resolved version of the STU model in \cite{KMPS}. The STU model is a smooth Calabi-Yau 3-fold, endowed with a map:
\[
\pi^{STU}:X^{STU}\rightarrow \pp^1.
\]
It arises as an anti-canonical section of a toric 4-fold $Y$. The fan datum for $Y$ can be found in \cite[Section 1.3]{KMPS}. We use their notation. There are 10 primitive rays $\{\rho_i; 1\leq i\leq 10\}$, and the corresponding divisors are denoted as $D_i\in\Pic(Y)$. The anti-canonical class is:
\[
-K_Y=\sum_{i=1}^{10}D_i.
\]

The general fiber of $\pi^{STU}$ is a smooth elliptic K3 surface, but there are 528 singular fibers \cite[Proposition 1]{KMPS}, each of which has exactly one ordinary double point singularity.  Let $\epsilon:C\rightarrow \p^1$ be a double cover branched along the 528 points corresponding to the singular fibers.
The pullback of $X^{STU}$ by $\epsilon$ has double point singularities, and by resolving them we obtain the resolved STU model:
\[
\tilde{\pi}^{S T U}: \widetilde{X}^{S T U} \rightarrow C.
\]
All fibers of $\tilde{\pi}^{STU}$ are smooth elliptic K3 surfaces. Moreover the toric divisors $D_5,D_3\in\Pic(Y)$ restrict to the universal section and fiber for $\tilde{\pi}^{STU}$. The family $\tilde{\pi}^{S T U}$ defines a curve in the moduli space $\F_U$:
\[
\iota:C\rightarrow \F_U.
\]

The intersection number $\iota^*(\lambda)$ is computed in \cite[Proposition 2]{KMPS}:
\[
\iota^*(\lambda)=4 E_{4}(q) E_{6}(q)[0]=4,
\]
where $E_4$ and $E_6$ are Eisenstein series, and we take the coefficient of $q^0$.

For the $\kappa$-classes, it suffices to perform the computation over the non-resolved STU model.  Since the tautological classes we consider are all invariant, we may assume the universal line bundles on $\F_U$ pull back to the toric divisors $D_5,D_3$. For $\kappa_{3,0,0}$, we have:
\[
\iota^*(\kappa_{3,0,0})=2\cdot\pi^{STU}_*\left(D_5^3\cdot\sum_{i=1}^{10}D_i\right),
\]
where the factor of 2 comes from the double cover $\epsilon$. Using toric geometry, all monomials of the form $D_i\cdot D_j\cdot D_k\cdot D_l$ can be explicitly determined. We obtain:
\[
\iota^*(\kappa_{3,0,0})=16.
\]
Other intersection numbers can be computed analogously. We record the final answers:
\[
\iota^*(\kappa_{3,0,0})=16\quad \iota^*(\kappa_{1,0,1})=-8\quad \iota^*(\kappa_{2,1,0})=-4
\]
\[\iota^*(\kappa_{0,1,1})=48\quad \iota^*(\kappa_{1,2,0})=0\quad \iota^*(\kappa_{0,3,0})=0.
\]
The four identities in the proposition then follow immediately.
\end{proof}

\bibliographystyle{alpha}
\bibliography{bibliography}

\end{document}